\documentclass[12pt]{amsart}

\usepackage[hmargin=1.3in,vmargin=1.3in]{geometry}

\usepackage{amsmath,amsthm,amssymb,amsfonts,verbatim}

 \providecommand{\F}{\mathbb{F}}

\RequirePackage[colorlinks,pagebackref,linkcolor=blue,filecolor = blue,citecolor = blue, urlcolor =blue]{hyperref}

\title[Folded AG Codes using Class Fields]{Optimal rate algebraic list decoding using narrow ray class fields}

\author[V. Guruswami and C. Xing]{Venkatesan Guruswami \and Chaoping Xing}

\thanks{Research of V.G supported
  in part by NSF grant CCF-0963975, and a Packard Fellowship.}

  \address{Computer Science Department,
  Carnegie Mellon University, Pittsburgh, PA 15213, USA.}

\email{guruswami@cmu.edu}

\address{Division of Mathematical Sciences, School of Physical \&  Mathematical Sciences, Nanyang Technological University, Singapore.}

\email{xingcp@ntu.edu.sg}

\newtheorem{lemma}{Lemma}[section]
\newtheorem{theorem}[lemma]{Theorem}
\newtheorem{prop}[lemma]{Proposition}

\newtheorem{defn}{Definition}
\newtheorem{example}[lemma]{Example}

\newtheorem*{open}{Open Problem}

\theoremstyle{remark}
\newtheorem{rmk}{Remark}

\newcommand{\eps}{\varepsilon}
\renewcommand{\epsilon}{\varepsilon}
\renewcommand{\le}{\leqslant}
\renewcommand{\ge}{\geqslant}

\newcommand{\vnote}[1]{}
\newcommand{\cpnote}[1]{}

\def\ZZ{\mathbb{Z}}
\def\PP{\mathbb{P}}

\def \mD {\mathcal{D}}

\def \mL {\mathcal{L}}

\def \mQ {\mathcal{Q}}
\def \Q {\mathcal{Q}}

\def\cL{{\mathcal L}}

\def\Supp{{\rm Supp}}

\newcommand{\Ga}{\alpha}
\newcommand{\Gb}{\beta}

\newcommand{\Ge}{\epsilon}

\newcommand{\Gk}{\kappa}
\newcommand{\Gl}{\lambda}    
     
\newcommand{\Gs}{\sigma}

\newcommand{\s}{\sigma}

\def\FH{\mathsf{F}}

\def \by {{\bf y}}

\def\supp {{\rm supp }}

\def \Gal {{\rm Gal}}
\def \Cl {{\rm Cl}}
\def \Fr {{\rm Fr}}
\def \Prin {{\rm Prin}}
\def\sgn {{\rm sgn}}
\def\Aut {{\rm Aut}}

\begin{document}

\maketitle
\thispagestyle{empty}

\begin{abstract}

We use class field theory, specifically Drinfeld modules of rank $1$, to construct a family of asymptotically good algebraic-geometric (AG) codes over fixed alphabets. Over a field of size $\ell^2$, these codes are within $2/(\sqrt{\ell}-1)$ of the Singleton bound. The functions fields underlying these codes are subfields with a cyclic Galois group of the narrow ray class field of certain function fields. The resulting codes are ``folded" using a generator of the Galois group. This generalizes earlier work by the first author on folded AG codes based on cyclotomic function fields. Using the Chebotarev density theorem, we argue the abundance of inert places of large degree in our cyclic extension, and use this to devise a linear-algebraic algorithm to list decode these folded codes up to an error fraction approaching $1-R$ where $R$ is the rate. The list decoding can be performed in polynomial time given polynomial amount of pre-processed information about the function field.

Our construction yields algebraic codes over constant-sized alphabets that can be list decoded up to the Singleton bound --- specifically, for any desired rate $R \in (0,1)$ and constant $\eps > 0$, we get codes over an alphabet size $(1/\eps)^{O(1/\eps^2)}$ that can be list decoded up to error fraction $1-R-\eps$ confining close-by messages to a subspace with $N^{O(1/\eps^2)}$ elements.
Previous results for list decoding up to error-fraction $1-R-\eps$ over constant-sized alphabets were either based on concatenation or involved taking a carefully sampled subcode of algebraic-geometric codes. In contrast, our result shows that these folded algebraic-geometric codes {\em themselves} have the claimed list decoding property.

\end{abstract}

\section{Introduction}

Reed-Solomon codes are a classical and widely used family of
error-correcting codes. They encode messages, which are viewed as
polynomials $f \in \F_q[X]$ of degree $< k$ over a finite field
$\F_q$, into {\em codewords} consisting of the evaluations of $f$ at a
sequence of $n$ distinct elements $\alpha_1,\dots,\alpha_n \in \F_q$
(this requires a field size $q \ge n$). We refer to $n$ as the block lengh of the code. The rate of this code, equal
to the ratio of number of message symbols to the number of codeword
symbols, equals $R = k/n$. Since two distinct polynomials of degree $<
k$ can agree on at most $k-1$ distinct points, every pair of
Reed-Solomon codewords differ on more than $n-k$ positions. In other words, the {\em relative distance} of this code, or the minimum fraction of positions two distinct codewords differ on, is bigger than $(1-R)$.
This means
that even if up to a fraction $(1-R)/2$ of
the $n$ codeword symbols, are corrupted in an {\em arbitrary} manner,
the message polynomial $f$ is still uniquely determined. Moreover,
classical algorithms, starting with \cite{peterson}, can recover the
message $f$ in such a situation in polynomial time.

For a fraction of errors exceeding $(1-R)/2$, unambiguous decoding of
the correct message is not always possible. This holds not just
for the Reed-Solomon code but for {\em every} code. However, if we
allow the decoder to output in the worst-case a small list of messages
whose encodings are close to the corrupted codeword, then it turns out
that one can correct a much larger error fraction. This model is
called {\em list decoding}. Using the probabilistic method, for any
$\eps > 0$, one can prove the abundance of codes of rate $R$ which can
be list decoded up to an error fraction $(1-R-\eps)$ with a maximum
output list size bounded by a constant depending only on $\eps$. This
error fraction is twice the classicial $(1-R)/2$ bound, and further is
optimal as the message has $Rn$ symbols of information and recovering it
up to some small ambiguity is impossible from fewer than a fraction $R$
of correct codeword symbols.

Recent progress in algebraic coding theory has led to the construction
of explicit codes which can be efficiently list decoded up to an error
fraction approaching the $1-R$ information-theoretic limit. The first
such construction, due to Guruswami and Rudra \cite{GR-FRS}, was {\em
  folded Reed-Solomon codes}. In the {\em $m$-folded} version of this
code (where $m$ is a positive integer), the Reed-Solomon (RS) encoding
$(f(1),f(\gamma),\cdots,f(\gamma^{n-1}))$ of a low-degree polynomial
$f \in \F_q[X]$ is viewed as a codeword of length $N = n/m$ over the
alphabet $\F_q^m$ by blocking together successive sets of $m$
symbols. Here $\gamma$ is a primitive element of the field $\F_q$. The
alphabet size of the folded RS codes is $q^m > N^m$. To list decode
these codes up to an error fraction $1-R-\eps$, one has to choose $m
\approx 1/\eps^2$ which makes the alphabet size a larger polynomial in
the block length. In comparison, the probabilistic method shows the
existence of such list decodable codes over an alphabet size
$\exp(O(1/\eps))$, which is also the best possible asymptotic
dependence on $\eps$.

It is possible to bring down the alphabet size of folded RS codes by
concatenating them with appropriate optimal codes found by a
brute-force search, followed by symbol redistribution using an
expander~\cite{GR-FRS}. However, the resulting codes have a large
construction and decoding complexity due to the brute-force decoding
of the inner codes used in concatenation. Furthermore, these codes
lose the nice algebraic nature of folded RS codes which endows them
with other useful features like list recovery and soft decoding. It is
therefore of interest to find explicitly described algebraic codes
over {\em smaller} alphabets with list decoding properties similar to folded
RS codes.

Algebraic-geometric (AG) codes are a generalization of Reed-Solomon
codes based on algebraic curves which have $n \gg q$ $\F_q$-rational
points. These enable construction of RS-like codes with alphabet size
smaller than (and possibly even dependent of) the block length. Thus,
they provide a possible avenue to construct the analog of folded RS
codes over smaller alphabets.

The algebraic crux in list decoding folded RS codes was the identity
$f(\gamma X) \equiv f(X)^q \pmod {E(X)}$ for $E(X) = X^{q-1}-\gamma$
which is an irreducible polynomial over $\F_q$. Extending this to
other algebraic-geometric codes requires finding a similar identity in
the function field setting. As noted by the first
author~\cite{Gur-cyclo}, this can be achieved using Frobenius
automorphisms $\sigma$ in cyclic Galois extensions, and considering
the residue of $f^{\sigma}$ at a place of high degree in the function
field. Using certain subfields of cyclotomic function fields,
Guruswami~\cite{Gur-cyclo} was able to extend the folded RS list
decoding result of \cite{GR-FRS} and obtain folded algebraic-geometric
codes of rate $R$ list decodable up to error fraction $1-R-\eps$ over
an alphabet size $(\log N)^{O(1/\eps^2)}$. In other words, the
alphabet size was reduced to poly-logarithmic in the block length $N$
of the code.

\subsection{Our result}
The main result in this work is a construction of folded
algebraic-geometric codes which brings down the alphabet size to a
{\em constant} depending only on $\eps$. This is based on algebraic function
fields constructed via class field theory, utilizing Drinfeld
modules of rank $1$.

\begin{theorem}[Main]
Let $\ell$ be a square prime power and let $q =\ell^2$. For every $R \in (0,1)$, there is an infinite family of $\F_q$-linear algebraic-geometric codes of rate at least $R$ which has relative distance at least $1-R-2/(\sqrt{\ell}-1)$.

For every pair of integers $m \ge s \ge 1$, the $m$-folded version of
these codes (which is a code over alphabet $\F_q^m$) can be list
decoded from an error fraction
\[ \tau = \frac{s}{s+1} \biggl( 1 - \frac{m}{m-s+1} \Bigl( R + \frac{2}{\sqrt{\ell}-1} \Bigr) \biggr) \ , \]
outputting a subspace over $\F_q$ with at most
$O(N^{(\sqrt{\ell}-1)s})$ elements that includes all message functions
whose encoding is within Hamming distance $\tau N$ from the
input. (Here $N$ denotes the block length of the code.)

Given a
polynomial amount of pre-processed information about the code, the
algorithm essentially consists of solving two linear systems over
$\F_q$, and thus runs in deterministic polynomial time.
\end{theorem}

Picking suitable parameters in the above theorem, specifically $\ell
\approx 1/\eps^2$, $s \approx 1/\eps$, and $m \approx 1/\eps^2$, leads
to folded AG codes with alphabet size $(1/\eps)^{O(1/\eps^2)}$ of any
desired rate $R \in (0,1)$ that are list decodable up to error
fraction $1-R-\eps$ with a maximum output list size bounded by
$N^{O(1/\eps^2)}$. In other words, the polylogarithmic alphabet size
of cyclotomic function fields is further improved to a constant
depending only on $\eps$.

We prove the above theorem by employing the recently developed {\em
  linear-algebraic} approach to list decoding, which was first used to
an alternate, simpler proof of the list decodability of folded RS codes up to
error fractions approaching $1-R$ (see \cite{GW-merged}).

One of the simple but key observations that led to this work is the
following. In order to apply the linear algebraic list decoder for a
folded version of AG codes (such as the cyclotomic function field based
codes of \cite{Gur-cyclo}), one can use the Frobenius automorphism
based argument to just combinatorially {\em bound} the list size, but
such an automorphism is {\em not} needed in the actual decoding
algorithm. In particular, we {\em don't} need to find high degree
places with a specific Galois group element as its Frobenius
automorphism (this was one of the several challenges in the cyclotomic
function field based construction~\cite{Gur-cyclo}), but only need the
{\em existence} of such places.  This allows us to devise a
linear-algebraic list decoder for folded versions of a family of AG
codes, once we are able to construct function fields with certain
stipulated properties (such as many rational places compared to the
genus, and the existence of an automorphism which powers the residue
of functions modulo some places). We then construct function
fields with these properties over a fixed alphabet using class field
theory, which is our main technical contribution.

This gives the first construction of folded AG codes over
constant-sized alphabets list decodable up to the optimal $1-R$ bound,
although we are not able to efficiently construct the (natural)
representation of the code that is utilized by our polynomial time
decoding algorithm. This representation consists of the evaluations of
 regular functions at  the rational places used
for encoding (by a regular function at a place, we mean a function having no pole at this place).

In our previous works \cite{GX-stoc12,GX-gabidulin}, we considered
list decoding of folded AG codes and a variant where rational points
over a subfield are used for encoding. We were able to show that a
{\em subcode} of these codes can be efficiently list decoded up to the
optimal $1-R-\eps$ error fraction. The subcode is picked based on
variants of subspace-evasive sets (subsets of the message space
that have small intersection with low-dimensional subspaces) that are
not explicitly constructed. In contrast, in this work we are able to
list decode the folded AG codes {\em themselves}, and no randomly
constructed subcode is needed.

\subsection{Techniques}

Our main techniques can be summarized as follows.

Our principal algebraic construction is that of an infinite family of
  function fields over a fixed base field $\F_q$ with many rational
  places compared to their genus, together with certain additional properties
  needed for decoding. Our starting point is a family of function
  fields $E/\F_\ell$ (where $\ell = \sqrt{q}$) such as those from the
  Garcia-Stichtenoth towers~\cite{GS95,GS96} which attain the
  Drinfeld-Vl\u{a}dut bound (the best possible trade-off between
  number of rational places and genus).
We consider the constant field extension $L = \F_q \cdot E$, and
 take its narrow ray class field of
 with respect to some high degree place. We descend to a carefully constructed subfield $F$ of this class field in which the $\F_q$-rational places in $L$ split completely, and further the extension $F/L$ has a cyclic Galois group.

A generator $\sigma$ of this cyclic group $\Gal(F/L)$, which is an
automorphism of $F$ of high order, is used to order the evaluation
points in the AG code and then to fold this code. This last part is
similar to the earlier cyclotomic construction, but there the full
extension $F/\F_q(X)$ was cyclic. This was a stringent constraint that
in particular ruled out asymptotically good function fields --- in
fact even abelian extensions must have the ratio of the number of
rational places to genus tend to 0 when the genus grows~\cite{FPS}. In
our construction, only the portion $F/L$ needs to be cyclic, and this
is another insight that we exploit.


Next, using the Chebotarev density theorem, we argue the existence of
many large degree places which are inert in the extension $F/L$ and
have $\sigma$ as their Frobenius automorphism. This suffices to argue
that the list size is small using previous algebraic
techniques. Essentially the values of the candidate message functions
at the inert places mentioned above can be found by finding the roots
of a univariate polynomial over the residue field, and these values
can be combined via Chinese remaindering to identify the message
function.

Under the linear-algebraic approach, the above list will in fact be a
subspace. Thus knowing that this subspace has only polynomially many
elements is enough to list all elements in the subspace in polynomial
time by solving a linear system! To solve the linear system, we make use of
the local power series expansion of a basis of the Riemann-Roch
message space at certain rational places of $F$ (namely those lying above a
rational place of $L$ that splits completely in $F/L$).

%
%

To summarize, some of the novel aspects of this work are:
\begin{enumerate}
\itemsep=1ex
\item The use of class fields based on rank one Drinfeld modules to construct function fields $F/\F_q$ with many $\F_q$-rational places compared to its genus, {\em and} which have a subfield $L$ such that $F/L$ is a cyclic Galois extension of sufficiently high degree.
\item The use of the Chebotarev density theorem to combinatorially bound the list size.
\item Decoupling the {\em proof} of the combinatorial bound on list
  size from the algorithmic task of {\em computing} the list. This
  computational part is tackled by a linear-algebraic decoding
  algorithm whose efficiency automatically follows from the list size
  bound.
\end{enumerate}

\subsection{Organization}
In Section 2, we show a construction of folded algebraic-geometric codes over arbitrary function fields with many rational places and an automorphism of relatively large order. Then we present a linear-algebraic list decoding of the folded codes. Under some assumption about the base function fields, we prove in the same section that the folded codes is deterministically list decodable up to the Singleton bound. Section 3 is devoted to the construction of the base function fields needed in Section 2 for constructing our folded codes. Our construction of the base function fields is through class field theory, specifically Drinfeld modules of rank $1$. In Section 4, we discuss the encoding and decoding of our folded codes by some possible approach of finding explicit equations of the base function fields that are constructed in Section 3. The main result of this paper is then stated after discussion of encoding and decoding.

\section{Linear-Algebraic List Decoding of Folded Algebraic-Geometric Codes}
\label{sec:ALD}
In this section, we first present a construction of folded algebraic geometric codes over arbitrary function fields with certain properties and then give a deterministic list decoding of folded algebraic geometric codes over certain function fields satisfying some conditions.

\subsection{Preliminaries on Function Fields}
For convenience of the reader, we start with some background on global function fields over finite fields.

For a prime power $q$, let $\F_q$ be the finite field of $q$ elements. An { algebraic function field}  over
$\F_q$ in one variable is a field extension $F \supset \F_q$ such
that $F$ is a finite algebraic extension of  $\F_q(x)$ for some
$x\in F$ that is transcendental over $\F_q$. The field $\F_q$ is called the full constant field of $F$ if the algebraic closure of $\F_q$  in $F$ is $\F_q$ itself. Such a function field is also called a global function field. From now on, we always denote by  $F/\F_q$  a function field $F$ with the full constant field $\F_q$.

Let $\PP_F$ denote the set of places of $F$. The divisor group, denoted by ${\rm Div}(F)$, is the free abelian group generated by all places in $\PP_F$. An element $G=\sum_{P\in\PP_F}n_PP$ of ${\rm Div}(F)$ is called a divisor of $F$, where $n_P=0$ for almost all $P\in\PP_F$. The support, denoted by $\Supp(G)$, of $G$ is the set $\{P\in\PP_F:\; n_P\neq 0\}$. For a nonzero function $z\in F$, the principal divisor  of $z$ is defined to be
${\rm div}(z)=\sum_{P\in\PP_F}\nu_P(z)P$, where $\nu_P$ denotes the normalized discrete valuation at $P$. The zero and pole divisors of $z$ are defined to be ${\rm div}(z)_0=\sum_{\nu_P(z)>0}\nu_P(z)P$ and ${\rm div}(z)_{\infty}=-\sum_{\nu_P(z)<0}\nu_P(z)P$, respectively.

For a divisor $G$ of $F$, we define the Riemann-Roch space associated with $G$ by
\[\mL(G):=\{f\in F^*:\; {\rm div}(f)+G\ge 0\}\cup\{0\}.\]
Then $\mL(G)$ is a finite dimensional space over $\F_q$ and its
dimension $\ell(G)$ is determined by the Riemann-Roch theorem which
gives
\[\ell(G)=\deg(G)+1-g+\ell(W-G),\]
where $g$ is the genus of $F$ and $W$ is a canonical divisor of degree $2g-2$. Therefore, we
always have that $\ell(G)\ge \deg(G)+1-g$ and the equality holds
if $\deg(G)\ge 2g-1$.

For a function $f$ and a place $P\in\PP_F$ with $\nu_P(f)\ge 0$, we denote by $f(P)$ the residue class of $f$ in the residue class field $F_P$ at $P$. For an automorphism $\phi\in \Aut(F/\F_q)$ and a place $P$, we denote by $P^{\phi}$ the place $\{\phi(x):\; x\in P\}$. For a function $f\in F$, we denote by $f^{\phi}$ the action of $\phi$ on $f$. If $\nu_P(f)\ge 0$ and $\nu_{P^{\phi}}(f)\ge 0$, then one has that $\nu_P(f^{\phi^{-1}})\ge 0 $ and $f(P^{\phi})=f^{\phi^{-1}}(P)$. Furthermore, for a divisor $G=\sum_{P\in\PP_F}m_PP$ we denote by $G^{\phi}$ the divisor $\sum_{P\in\PP_F}m_PP^{\phi}$.

\subsection{Folded Algebraic Geometric Codes}
 To construct our folded codes, we assume that there exists  a global function field $F$ with the full constant field $\F_q$ having the following property:
 \begin{center}
{\bf Property (P1)}
 \end{center}
\begin{itemize}
\item [(i)]  There exists  an automorphism $\Gs$ in $\Aut(F/\F_q)$;
\item [(ii)] $F$ has $mN$ distinct rational places $P_1, P_1^{\Gs},\dots,  P_1^{\Gs^{m-1}}, P_2, P_2^{\Gs},\dots,$ $  P_2^{\Gs^{m-1}}, \dots, $ $P_N, P_N^{\Gs},\dots,  P_N^{\Gs^{m-1}}$;
    \item [(iii)] $F$ has a divisor $D$ of degree $e$ such that $D$ is fixed under $\Gs$, i.e., $D^{\Gs}=D$; and $P_i^{\Gs^j}\not\in\Supp(D)$ for all $1\le i\le N$ and $0\le j\le m-1$.
\end{itemize}

A folded algebraic geometric code can be defined as follows.
\begin{defn}[Folded AG codes]
\label{def:f-ag-code}{\rm
The  folded  code from $F$ with parameters $N,l,q,e,m$, denoted by ${\FH}(N,l,q,e,m)$,  encodes a message function $f \in \cL(lD)$  as
\begin{equation}
\label{eq:f-ag-defn} \pi:\quad
f \mapsto
\left(
\left[\begin{array}{c}  f(P_1) \\  f(P_1^{\s}) \\ \vdots \\  f(P_1^{\s^{m-1}})\end{array}\right],
\left[\begin{array}{c} f(P_2) \\  f(P_2^{\s}) \\ \vdots \\  f(P_2^{\s^{m-1}})\end{array}\right],
\ldots,
\left[\begin{array}{c} f(P_N) \\  f(P_N^{\s}) \\ \vdots \\  f(P_N^{\s^{m-1}})\end{array}\right]
\right)  \in \left( \F_{q}^m \right)^{N} \ .
\end{equation}
}\end{defn}
Note that the folded code ${\FH}(N,l,q,e,m)$ has the alphabet $\F_q^m$ and it is $\F_q$-linear. Furthermore, ${\FH}(N,l,q,e,m)$ has the following parameters.
\begin{lemma}\label{lem:para} If $le<mN$, then
the above code ${\FH}(N,l,q,e,m)$ is an $\F_q$-linear code with alphabet size $q^{m}$, rate at least $\frac{le-g+1}{Nm}$, and minimum distance at least $N  - \frac{le}{m}$.
\end{lemma}
\begin{proof} It is clear that the map $\pi$ in  (\ref{eq:f-ag-defn}) is $\F_q$-linear and  the kernel of $\pi$ is
\[ \cL\Bigl( lD-\sum_{i=1}^N\sum_{j=0}^{m-1}P_i^{\s^j}\Bigr) \]
 which is $\{0\}$ under the condition that $le<mN$. Thus, $\pi$ is injective. Hence, the rate is at least $\frac{le-g+1}{Nm}$ by  the Riemann-Roch theorem. To see the minimum distance, let $f$ be a nonzero function in $\cL(lD)$ and assume that $I$ is the support of $\pi(f)$. Then the Hamming weight ${\rm wt}_H(\pi(f))$ of $\pi(f)$ is $|I|$ and  $f\in \cL\left(lD-\sum_{i\not\in I}\sum_{j=0}^{m-1}P_i^{\s^j}\right)$. Thus, $0\le \deg\left(lD-\sum_{i\not\in I}\sum_{j=0}^{m-1}P_i^{\s^j}\right)=le-m(N-|I|)$, i.e., ${\rm wt}_H(\pi(f))=|I|\ge N-\frac{le}{m}$. This completes the proof.
\end{proof}
\subsection{List Decoding of Folded Algebraic Geometric Codes}

Suppose a codeword (\ref{eq:f-ag-defn}) encoded from $f\in  \cL(lD)$ was transmitted and received as
\begin{equation}
\label{eq:recd-word}
\mathbf{y} =
\left(
\begin{array}{ccccc}
y_{1,1} & y_{2,1} & & & y_{N,1}\\
y_{1,2} & y_{2,2} & & & \vdots\\
& & & \ddots & \\
 y_{1,m} & \cdots &&& y_{N,m}
\end{array}
\right),
\end{equation}
where some columns are erroneous.
Let $s \ge 1$ be an integer parameter associated with the decoder.
\begin{lemma}
\label{lem:herm-interpolation}
Given a received word as in {\rm (\ref{eq:recd-word})},  we can find a nonzero linear polynomial in $F[Y_1,Y_2,\dots,Y_s]$ of the form
\begin{equation}
\label{eq:form-of-Q}
Q(Y_1,Y_2,\dots,Y_s) =
A_0 + A_1 Y_1 + A_2 Y_2 + \cdots + A_s Y_s \
\end{equation}
satisfying
\begin{equation}
\label{eq:interpolation-cond}
Q(y_{i,j+1},y_{i,j+2},\cdots,y_{i,j+s}) =A_0(P_i^{\s^{j}})+A_1(P_i^{\s^{j}})y_{i,j+1}+\cdots+A_s(P_i^{\s^{j}})y_{i,j+s}= 0
\end{equation}
{for } $i=1,2,\dots,N$ { and }  $j =0,1,\dots,m-s$. The coefficients $A_i$ of $Q$ satisfy
$A_i \in \cL(\kappa D)$ for $i=1,2,\dots,s$ and $A_0\in \cL((\Gk+l)D)$ for a ``degree" parameter $d$ chosen as
\begin{equation}
\label{eq:choice-of-kappa}
 \Gk=\left\lfloor \frac {N(m-s+1)-el+(s+1)(g-1)+1}{e(s+1)}\right\rfloor .
\end{equation}
\end{lemma}
\begin{proof} Let $u$ and $v$ be dimensions of $\cL(\Gk D)$ and $\cL((\Gk+l)D)$, respectively. Let $\{x_1,\dots,x_u\}$ be an $\F_q$-basis of $\cL(\Gk D)$ and extend it to an $\F_q$-basis $\{x_1,\dots,x_v\}$ of  $\cL((d+l)D)$. Then $A_i$ is an $\F_q$-linear combination of $\{x_1,\dots,x_u\}$ for $i=1,2,\dots,s$ and $A_0$ is an $\F_q$-linear combination of $\{x_1,\dots,x_v\}$. Determining the functions $A_i$ is equivalent to determining the coefficients in the combinations of $A_i$. Thus, there are totally $su+v$ freedoms to determine $A_0,A_1,\dots,A_s$. By the Riemann-Roch theorem, the number of freedoms is at least $s(\Gk e-g+1)+(\Gk+l)e-g+1$.

On the other hand, there are totally $N(m-s+1)$ equations in (\ref{eq:interpolation-cond}). Thus, there must be one nonzero solution by the condition (\ref{eq:choice-of-kappa}), i.e., $Q(Y_1,Y_2,\dots,Y_s)$ is a nonzero polynomial.
\end{proof}

\begin{lemma}
\label{lem:Q-is-good}
If $f$ is a function in $\cL(lD)$ whose encoding (\ref{eq:f-ag-defn}) agrees
with the received word $\mathbf{y}$ in at least $t$ columns with
\[ t>\frac{(\Gk+l)e}{m-s+1} \ , \]
 then $Q(f,f^{\s^{-1}},\dots,f^{\s^{-(s-1)}})$ is the zero function, i.e., \begin{equation}\label{eq:alg-eqn}A_0+A_1f+A_2f^{\s^{-1}}+\cdots+A_sf^{\s^{-(s-1)}}=0.\end{equation}
\end{lemma}
\begin{proof}  Since $D=D^{\s}$, we have $f^{\s^i}\in \cL(lD)$ for all $i\in \ZZ$. Thus, it is clear that $Q(f,f^{\s^{-1}},\dots,f^{\s^{-(s-1)}})$ is a function in $\cL((\Gk+l)D)$.

Let us assume that $I\subseteq\{1,2,\dots,N\}$ is the index set such that the $i$th columns  of $\by$ and $\pi(f)$ agree if and only if $i\in I$.  Then we have $|I|\ge t$. For every $i\in I$ and $0\le j\le m-s$, we have by (\ref{eq:interpolation-cond})
\begin{eqnarray*}
0&=&A_0(P_i^{\s^{j}})+A_1(P_i^{\s^{j}})y_{i,j+1}+A_2(P_i^{\s^{j}})y_{i,j+2}+\cdots+A_s(P_i^{\s^{j}})y_{i,j+s}\\
&=&A_0(P_i^{\s^{j}})+A_1(P_i^{\s^{j}})f(P_i^{\s^{j}})+A_2(P_i^{\s^{j}})f(P_i^{\s^{j+1}}))+\cdots+A_s(P_i^{\s^{j}})f(P_i^{\s^{j+s-1}})\\
&=&A_0(P_i^{\s^{j}})+A_1(P_i^{\s^{j}})f(P_i^{\s^{j}})+A_2(P_i^{\s^{j}})f^{\s^{-1}}(P_i^{\s^{j}})+\cdots+A_s(P_i^{\s^{j}})f^{\s^{-s+1}}(P_i^{\s^{j}})\\
&=&\left(A_0+A_1f+A_2f^{\s^{-1}}+\cdots+A_sf^{\s^{-s+1}}\right)(P_i^{\s^{j}}),\end{eqnarray*}
 i.e., $P_i^{\s^{j}}$ is a zero of $Q(f,f^{\s},\dots,f^{\s^{s-1}})$. Hence, $Q(f,f^{\s^{-1}},\dots,f^{\s^{-(s-1)}})$ is a function in $\cL\left((\Gk+l)D-\sum_{i\in I}\sum_{j=0}^{m-s}P_i^{\s^{j}}\right)$. Our desired result follows from the fact that $\deg\left((\Gk+l)D-\sum_{i\in I}\sum_{j=0}^{m-s}P_i^{\s^{j}}\right)<0$.
\end{proof}


%
%

By Lemma \ref{lem:Q-is-good}, we know that all candidate functions $f$ in our list must satisfy the equation (\ref{eq:alg-eqn}). In other words,  we have to study the solution set of  the equation (\ref{eq:alg-eqn}). In our previous work \cite{GX-stoc12}, to upper bound the list size, we analyzed the solutions of the equation (\ref{eq:alg-eqn}) by considering local expansions at a certain point. This local expansion method only guarantees  a structured list of exponential size. Through precoding by using the structure in the list, we were able to obtain a Monte Carlo construction of subcodes of these codes with polynomial time list decoding.
The other method used in \cite{GR-FRS} for decoding the Reed-Solomon codes is to construct an irreducible polynomial $h(x)$ of degree $q-1$ such that every polynomial $f$ satisfies $f^{\s^{-1}}\equiv f^{q} \mod{h} $. Then the solution set of (\ref{eq:alg-eqn}) is the same as the solution set of the equation $ A_0+A_1f+A_2f^{q}+\cdots+A_sf^{q^{s-1}} \equiv 0 \mod{h}$ since $\deg(f)<q-1=\deg(h)$. Thus, there are at most $q^{s-1}$ solutions for the equation (\ref{eq:alg-eqn}). In order to generalize the latter idea used for the Reed-Solomon code to upper bound our list size of our folded algebraic geometric codes, we require some further property that $F$ must satisfy.
\begin{center}
{\bf Property (P2)}
 \end{center}
\begin{itemize}
\item [(i)]  There exists  a finite set $T$ of places of $F$ such that $\supp(D)\cap T=\emptyset$ and every place in $T$ has the same degree;
\item [(ii)] There exists an integer $u> 0$ such that $f^{\s^{-1}}\equiv f^{q^u} \mod{R} $, i.e., $f^{\s^{-1}}(R)\equiv f(R)^{q^u}$ for every $R\in T$ and all $f\in F$ with $\nu_R(f)\ge 0$.
    \item [(iii)] $\sum_{R\in T}\deg(R)>le$.
\end{itemize}

\begin{lemma}\label{lem:first-list-size}
Assume that $F$ satisfies (P1) and (P2), then the solution set of the equation {\rm (\ref{eq:alg-eqn})} has size at most $q^{u(s-1)|T|}$.
\end{lemma}
\begin{proof} Consider the map $\psi:\; \cL(lD)\rightarrow \prod_{R\in T}F_R$ by sending $z$ to $\psi(z)=(z(R))_{R\in T}$. It is clear that $
 \psi$ is $\F_q$-linear. Furthermore, $\psi$ is injective. Indeed, if $\psi(y)=\psi(z)$ for some $y,z\in\cL(lD)$, then $\psi(y-z)=0$, i.e., $(y-z)(R)=0$ for all $R\in T$. Hence, $y-z$ belongs to $\cL(lD-\sum_{R\in T}R)$. So, we must have $y-z=0$ since $\deg(lD-\sum_{R\in T}R)=le-\sum_{R\in T}\deg(R)<0$.

Let $W$ be the solution set of (\ref{eq:alg-eqn}). Then  for every $R\in T$ and $f\in W$,  we have
\begin{eqnarray*}
0&=&A_0(R)+A_1(R)f(R)+A_2(R)f^{\s^{-1}}(R)+\cdots+A_s(R)f^{\s^{-(s-1)}}(R)\\
&=&A_0(R)+A_1(R)f(R)+A_2(R)f^{q^u}(R)+\cdots+A_s(R)f^{q^{u(s-1)}}(R)\in F_R.
\end{eqnarray*}
The above equation has at most $q^{u(s-1)}$ solutions in $F_R$. This implies that the set $W_R:=\{f(R):\; f\in W\}\subseteq F_R$ has size at most $q^{u(s-1)}$. Moreover, it is clear that $\psi(W)\subseteq \prod_{R\in T}W_R$. Thus, our desired result follows from
\[|W|=|\psi(W)|\le \left|\prod_{R\in T}W_R\right|=\prod_{R\in T}|W_R|\le q^{u(s-1)|T|}.\]
This completes the proof.
\end{proof}

\begin{rmk} {\begin{itemize} \item[(i)] From the proof of Lemma \ref{lem:first-list-size}, we can see that the places in the set $T$  given in (P2)(i) need not all be of the same degree. In fact, as long as the condition in (P2)(iii), i.e.,  $\sum_{R\in T}\deg(R)>le$ is satisfied, we  can guarantee the list size given in Lemma \ref{lem:first-list-size}.
\item[(ii)] In \cite{Gur-cyclo}, the set $T$ given in (P2)(i) has a single place with degree bigger than $le$. Then the list size would be $q^{u(s-1)|T|}=q^{u(s-1)}$. It seems that we could get a smaller list size. However, it is not possible in our case. The reason is that if we choose $|T|=1$, then the degree of the place $R$ is bigger than $le$ and thus the extension degree of $F/L$ is at least $le/u$, where $L$ is the subfield of $F$ fixed by $\langle\Gs\rangle$. On the other hand, we will see that the extension degree of $F/L$ is $e=|\langle\Gs\rangle|$ from (P3) below. This means that we must have $u> l$. This restriction on $u$ makes our list size even bigger when we choose a  set $T$ of single place.
\end{itemize}
}
\end{rmk}

Now we look at the fraction of errors that we can correct from the above list decoding. By taking $t=1+\left\lfloor\frac{(\Gk+l)e}{m-s+1}\right\rfloor$ and
combining Lemmas \ref{lem:Q-is-good} and \ref{lem:herm-interpolation}, we conclude the fraction of errors $\tau = 1-t/N$ satisfies
\begin{equation}
\label{eq:herm-error-frac}
 \tau \thickapprox \frac{s}{s+1}- \frac{s}{s+1}\times \frac m{m-s+1}\times\frac{k+g}{mN} , \
 \end{equation}
 where $k$ is the dimension of $\cL(lD)$ which is at least $le-g+1$.

Let $\ell$ be an even power of a prime and $q=\ell^2$.  In Section \ref{FF}, we will show that, for any given family $\{E/\F_{\ell}\}$ with $N(E/\F_{\ell})/g(E)\rightarrow\sqrt{\ell}-1$ and $g(E)\rightarrow\infty$, where $N(E/\F_{\ell})$ denotes the number of $\F_{\ell}$-rational places of $E$,   there exists  a family $\{F/\F_q\}$ of function fields satisfying the following
\begin{center}
{\bf Property (P3)}
 \end{center}
\begin{itemize}
\item [(i)]  $F/L$ is a cyclic Galois extension of degree $e$, where $L$ is the constant field extension $E\cdot\F_q$ and $e=(\ell^r+1)/(\ell+1)$ with $r= 2\lceil N(E/\F_{\ell})/(\sqrt{\ell}-1)\rceil +1$;
\item [(ii)] There exists a subset $S$ of $\PP_F$ such that $|S|\ge q^r$ and $\deg(R)=3re$ for all $R\in S$. Moreover, for any $R\in S$ and $z\in F$ with $\nu_R(z)\ge 0$, one has $z^{\s^{-1}}\equiv z^{q^{3r}} \mod{R}$, where $\s$ is the generator of the Galois group $\Gal(F/L)$ (note that $\Gal(F/L)$ is a subgroup of ${\rm Aut}(F/\F_q)$).
\item [(iii)] Every rational place of $E$ can be regarded as a rational place of $L$ and it splits completely in $F$. Thus, one has  $N(F/\F_q)\ge eN(E/\F_{\ell})$, where $N(F/\F_q$ denotes the number of $\F_q$-rational places of $F$. Furthermore,
\[ \liminf N(F/\F_q)/g(F)\ge (\sqrt{\ell}-1)/2=(q^{1/4}-1)/2 \ . \]
\end{itemize}

\begin{theorem}\label{thm:assum-list-size}
Let $\ell$ be a square prime power and let $q =\ell^2$. For every $R \in (0,1)$, there is an infinite family of   folded codes given in {\rm (\ref{eq:f-ag-defn})}  of rate at least $R$ which has relative distance at least $1-R-2/(\sqrt{\ell}-1)$.

For every pair of integers $m \ge s \ge 1$, these codes can be list
decoded from an error fraction
\[ \tau = \frac{s}{s+1} \biggl( 1 - \frac{m}{m-s+1} \Bigl( R + \frac{2}{\sqrt{\ell}-1} \Bigr) \biggr) \ , \]
outputting a subspace over $\F_q$ with at most
$O(N^{(\sqrt{\ell}-1)s})$ elements that includes all message functions
whose encoding is within Hamming distance $\tau N$ from the
input. (Here $N$ denotes the block length of the code.)
\end{theorem}

\begin{proof} Let $\{F/\F_q\}$ be a family of function fields satisfying (P3) constructed in Theorem  {\rm \ref{2.3}}. Choose a rational place $\infty$ of $E$ and regard it as a rational place of $L$. Define the divisor $D:=l\sum_{P_{\infty}|\infty, P_{\infty}\in\PP_F}P$.  Then it is easy to see that $D^{\Gs}=D$. For every rational place $P$ of $E$, there are exactly $e$ rational places
of $F$  lying over $P$ and they can be represented as $P, P^{\s},\dots, P^{\s^{e-1}}$.  By taking away those rational places
lying over $\infty$, we have at least $e(n-1)$ rational places of $F$, where $n=N(E/\F_{\ell})$. Thus,
for an integer $m$ with $1\le m < e$, we can label $Nm$ distinct places
$P_1, P_1^{\s},\dots, P_1^{\s^{ m- 1}},\dots,P_N, P_N^{\s},\dots, P_N^{\s^{ m -1}}$
of F such that none of them
lies over $\infty$, as long as $N \le (n-1) \lfloor\frac{e}m\rfloor =(N(E/\F_{\ell})-1)\lfloor \frac{e}m\rfloor$.

It is clear that the property (P1) is satisfied and hence we can define the folded algebraic geometric code ${\FH}(N,l,q,e,m)$ as in Definition \ref{def:f-ag-code}. We choose $l$   to satisfy the condition $le<Nm$.
Choose a subset $T$ of $S$ with $|T|=\lceil \sqrt{\ell}-1\rceil$. Then we have
\[\sum_{R\in T}\deg(R) \ge 3re(\sqrt{\ell}-1)\ge 6N(E/\F_{\ell})e=6mN(E/\F_{\ell})\frac em>Nm>le.\]
This implies that the property (P2) is also satisfied. Hence, the code ${\FH}(N,l,q,e,m)$ is deterministically list decodable with list size at most $q^{3r(s-1)\lceil \sqrt{\ell}-1\rceil}=O(q^{sn})$. Note that the code length is $N$ which is approximately $en=m = O(n\ell^{2n/(\sqrt{\ell}-1)}/\ell m)$. Thus, the list size is $O(N ^{( \sqrt{\ell}-1)s} )$.

The claimed error fraction follows from (\ref{eq:herm-error-frac}) and the fact that $g/Nm\rightarrow (\sqrt{\ell}-1)/2$.

\end{proof}

%

\section{Construction of a Family of Function Fields}\label{FF}
In view of Theorem \ref{thm:assum-list-size}, it is essential to construct a family of function fields satisfying the Property (P3). In this section, we use class field theory and the Chebotarev Density Theorem to show the existence of such a family.

\subsection{Narrow-ray class fields and Drinfeld module of rank one}
\label{subsec:narrow-ray}
Throughout this subsection, we fix a function field $F$ over $\F_q$ and a rational place $\infty$. Denote by $A$ the ring
\[A:=\{x\in F:\; \nu_P(x)\ge 0 \ \mbox{for all $P\not=\infty$}\}.\]
Let $\Fr$ and $\Prin$ denote the fractional ideal group and the principal ideal group of $A$, respectively.
Then the fractional idea class group $\Cl(A)=\Fr/\Prin$ of $A$ is actually isomorphic to the zero degree divisor class group of $F$.

Let $D=\sum_P\nu_P(D)P$ be a positive divisor of $F$ with $\infty\not\in {\rm supp(D)}$.
For $x\in F^*$,
$x\equiv 1 \, ({\rm mod} \ D)$
means that $x$ satisfies the following condition:
\begin{quote}
if $P\in\supp(D)$, then
$\nu_P(x-1)\ge \nu_P(D)$.
\end{quote}
 Let $\Fr_{D}$ be the
subgroup of $\Fr$ consisting of the fractional ideals of $A$ that are relatively prime
to $D$, that is,
$$\Fr_{D} = \{\Re\in \Fr: \nu_P(\Re) = 0 \ {\rm for \ all} \ P\in {\rm supp(D)}
\}.$$
Define the subgroup $\Prin_{D}$ of $\Fr_{D}$ by
$$\Prin_{D}=\{xA:x\in F^*, \, x\equiv 1 \, ({\rm mod} \ D)\}.$$
The factor group
$\Fr_{D}/\Prin_{D}$
is called the { $\infty$-ray class group} modulo $D$. It
is a finite group and denoted by $\Cl_D(A)$.
If $D=0$, then we obtain the fractional ideal class group $\Cl(A)$.

Choose a local parameter $t\in F$ at $P$, i.e., $\nu_P(t)=1$. Then the $\infty$-adic completion $\F_{\infty}$ of $F$ consists of all power series of the form $\sum_{i=v}^{\infty}a_it^i$, where $v\in\ZZ$ and $a_i\in\F_q$ for all $i\ge v$. We can define
a  sign function $\sgn$  from $\F_{\infty}^*$ to $\F_q^*$ by sending  $\sum_{i=v}^{\infty}a_it^i$ to $a_v$ if $a_v\neq 0$  (see \cite[pages 50-51]{NX01}). Define
$$\Prin_{D}^+=\{xA:\; x\in F^*, \ \sgn(x)=1, \, x\equiv 1 \, ({\rm mod} \ D)\}.$$
\begin{defn}[Narrow ray class group]
The factor group
\[\Cl^+_D(A) = \Fr_D/\Prin^+_D\]
is called the narrow ray class group of $A$ modulo $D$ {\rm (}with respect to the $\sgn${\rm )}.
%
When $D$ is supported on a single place $Q$, i.e., $D = 1 \cdot Q$, we denote $\Cl_D(A)$ (resp. $\Cl^+_D(A)$) as simply $\Cl_Q(A)$ {\rm (}resp. $\Cl^+_Q(A)${\rm )}.
\end{defn}

We have the following result~\cite[Proposition 2.6.4]{NX01} concerning narrow ray and ideal class groups.

\begin{lemma}\label{2.1.1}
\begin{itemize}
\item[{\rm (i)}] $\Prin_D^+$ is a subgroup of $\Prin(D)$ and $\Prin_D/\Prin^+_D\simeq \F_q^*.$
\item[{\rm (ii)}] We have the isomorphisms \[\Cl_D^+(A)/\F_q^*\backsimeq\Cl_D^+(A)/( \Prin_D/\Prin^+_D)\backsimeq \Cl_D(A).\]
\item[{\rm (iii)}] We have \[\Cl^+_D(A)/(A/\mD)^* \backsimeq\Cl(A),\]
where $\mD$ is the ideal of $A$ corresponding to the divisor $D$, i.e., $\mD=\prod\wp^{n_P}$ if $D=\sum n_P P$ with $\wp$ being the prime ideal of $A$ corresponding to the place $P$.
\end{itemize}
\end{lemma}

Let $H_A$ denote the Hilbert class field of $F$ with respect to the place $\infty$, i.e, $H_A$ is the maximal abelian extension in a fixed algebraic closure of $F$ such that $\infty$ splits completely. Then we have $\Gal(H_A/F)\backsimeq \Cl(A)$.

We will use the following result from class field theory (see \cite[Sections 2.5-2.6]{NX01}).
\begin{prop}
\label{prop:CFT}
Now let $Q$ be a place of degree $d>1$ in a function field $F/\F_q$. Then there exists an abelian extension $F^Q$ of $F$ (called a narrow ray class field) with the following properties:
\begin{itemize}
\item[(i)] $\Gal(F^Q/F)\simeq\Cl_Q^+(A)$ and the extension degree of $F^Q/F$ is $|\Cl_Q^+(A)|=(q^d-1) |\Cl(A)|=(q^d-1)h_F$, where $h_F:=|\Cl(A)|$ is the zero degree divisor class number of $F$.

\item[(ii)] The Hilbert class field $H_A$ of $F$ is a subfield of $F^Q$ and the Galois group $\Gal(F^Q/H_A)$ is isomorphic to $(A/\Q)^* \backsimeq\F^*_{q^d}$, where $\Q$ is the ideal of $A$ corresponding to the place $Q$.

\item[(iii)] $\infty$ and $Q$ are only ramified places in $F^Q/F$. The inertia group of $Q$ in $F^Q/F$ is $(A/{\mQ})^*$ and the inertia group of $\infty$ is $\F_q^*$. In particular, the ramification index of $Q$ is $e_Q=q^d-1$ and the ramification index of $\infty$ is $q-1$. Furthermore, $\infty$ splits into rational places in $F^Q$.


\item[(iv)] In the Galois extension $F^Q/F$, the Frobenius automorphism of a place $P$ that is different from $\infty$ and $Q$ is $P$ itself when $P$ is viewed as an element in $\Cl_Q^+(A)$.

\end{itemize}
\end{prop}

From the above, we can easily compute the genus of $F^Q$ by the Hurwitz genus formula, i.e.,
\[2g(F^Q)-2=(2g(F)-2)h_F(q^d-1)+(q-2)h_F\frac{q^d-1}{q-1}+d(q^d-2)h_F \ .\]

Next, we give a more explicit description of the narrow ray class field $F^Q$ in terms of a Drinfeld module of rank one.

Let $p$ be the characteristic of $\F_q$ and let $\pi: c\mapsto c^p$ be the Frobenius endomorphism of $H_A$.
Consider the left twisted polynomial ring $H_A[\pi]$ whose
elements are polynomials in $\pi$ with coefficients from $H_A$ written on the left; but
multiplication in $H_A[\pi]$ is twisted by the rule
$$\pi u=u^p\pi\qquad \mbox{for all} \ u\in H_A.$$
Let $\tilde{D} : H_A[\pi]\longrightarrow H_A$ be the map which assigns to each
 polynomial in $H_A[\pi]$
 its constant term.

\begin{defn}
\label{2.1.2}
{\rm A {Drinfeld} $A$-{module} of rank $1$ over $H_A$ is a  ring homomorphism $\phi:
A\longrightarrow H_A[\pi]$, $a\mapsto\phi_a$, such that:\par
(i) not all elements of $H_A[\pi]$ in the image of $\phi$ are
constant polynomials;\par
(ii) $\tilde{D} \circ \phi$ is the identity on $A$;\par
(iii) There exists a positive integer $\lambda$ such that $\deg(\phi_a)=-\lambda\nu_{\infty}(a)$ for all nonzero $a\in A$, where 
$\deg(\phi_a)$ is the degree of $\phi_a$ as a polynomial in $\pi$.}
\end{defn}

\begin{example}\label{ex:2.1.2} {\rm Consider the rational function field $F=\F_q(T)$ with
$\infty$ being the pole place of $T$. Then we have $A=\F_q[T]$ and $H_A=F=\F_q(T)$. A
Drinfeld $A$-module $\phi$ of rank $1$ over $F$ is uniquely
determined by the image $\phi_T$ of $T$. By Definition \ref{2.1.2} we
must have
$$\tilde{D}(\phi_T)=(\tilde{D}\circ\phi)(T)=T,$$
i.e., $\phi_T$ is a nonconstant polynomial in $\pi$ with the constant term $T$.
Since $\deg(\phi_T)=-\lambda\nu_{\infty}(T)=\lambda$, we know that $\phi_T$ is of  the form $T+f(\pi)\pi+x\pi^{\lambda}$ for an element $x\in F^*$ and $f(\pi)\in F[\pi]$ with $\deg(f(\pi))\le \lambda-2$. Taking $x=1$ and $f(\pi)=0$ gives the so-called { Carlitz module}, which yields the construction of cyclotomic function fields.}
\end{example}

\begin{defn} {\rm We fix a sign function sgn. We say that a Drinfeld
$A$-module $\phi$ of rank $1$ over $H_A$ is sgn-{normalized}  if
${\rm sgn}(a)$ is equal to the leading coefficient of $\phi_a$ for all $a\in A$. In particular, the leading coefficient of $\phi_a$ must belong to $\F^*_q$. }
\end{defn}

\begin{defn}[Twisted polynomials corresponding to an ideal]
Given a Drinfeld $A$-module $\phi$ of rank $1$ over $H_A$ and a  prime ideal $\Q$
of $A$, let $I_{\Q, \phi}$ be the left ideal generated in $H_A[\pi]$ by the twisted polynomials
$\phi_a$, $a\in \Q$. As left ideals are principal,
$I_{\Q, \phi}=H_A[\pi]\phi_{\Q}$ for a unique monic twisted polynomial
$\phi_{\Q}\in H_A[\pi].$
\end{defn}

Let $K$ be any $H_A$-algebra. Then for a polynomial $f(\pi)=\sum_{i=0}^k b_i\pi^i\in H_A[\pi]$
the action of $f(\pi)$ on $K$ is defined by
\[  f(\pi)(t)=\sum_{i=0}^k b_it^{p^i}\quad \mbox{for all } t\in K \ . \]
Let $\overline{H_A}$ denote a fixed algebraic closure of $H_A$ whose
additive group $(\overline{H_A},+)$ is equipped with an
$A$-module structure under the action of $\phi$.

\begin{defn} {\rm Let $\phi$ be a sgn-normalized Drinfeld $A$-module of rank $1$ over $H_A$ and $\Q$ be a  nonzero ideal of $A$. The $\Q$-{torsion module} $\Lambda_{\phi}(\Q)$ associated with $\phi$ is defined by
$$\Lambda_{\phi}(\Q)=\{t\in (\overline{H_A},+): \phi_{\Q}(t)=0\}.$$}
\end{defn}

The following are a few basic facts about $\Lambda_{\phi} (\Q)$:\par
(i) $\Lambda_{\phi} (\Q)$ is a finite set of cardinality $|\Lambda_{\phi} (\Q)|=
p^{{\rm deg}(\phi _{\Q})}$;\par
(ii) $\Lambda_{\phi} (\Q)$ is an $A$-submodule of $(\overline{H_A},+)$ and a cyclic $A$-module
isomorphic to $A/\Q $;\par
(iii) $\Lambda_{\phi} (\Q)$ has $\Phi (\Q):=|(A/\Q)^*|$
generators as a cyclic $A$-module, where $(A/\Q)^*$ is the
group of units of the ring $A/\Q $.

The elements of $\Lambda_{\phi}(\Q)$ are also called the $\Q$-{ torsion elements} in $(\overline{H_A},+)$.
The following gives an explicit description of narrow ray class fields in terms of extension fields obtained by adjoining these torsion elements.

\begin{prop}
\label{prop:explicit-CFT}
The extension field $H_A(\Lambda_{\phi}(\Q))$
obtained by adjoining these $\Q$-torsion elements to $H_A$ is isomorphic to the narrow ray class field $F^Q$ from Proposition {\rm \ref{prop:CFT}}, where $Q$ is the place corresponding to the ideal $\Q$.
\end{prop}

In the case where $F$ is the rational function field and $\phi$ is the Carlitz module in Example \ref{ex:2.1.2},  the field $F^Q$ is the cyclotomic function field
over  $F$ with modulus $\Q$.

\subsection{A family of function fields}
\label{sec:FF family}
In this subsection, we assume that  $\ell$ is a prime power and  $q=\ell^2$.
The following is the key technical component of our construction of the function fields needed for our list-decodable code construction.
\begin{lemma}\label{2.1.3} Let $E/\F_{\ell}$ be a function field with at least one rational point $\infty$ and a place $Q$ of degree $r$, where $r > 1$ is an odd integer. Then there exists a function field $F/\F_q$ such that
\begin{itemize}
\item[{\rm (i)}] $F/(\F_q\cdot E)$ is a cyclic abelian extension with $[F:\F_q\cdot E]=\frac{\ell^r+1}{\ell+1}$.
\item[{\rm (ii)}] $N(F/\F_q)\ge  \frac{\ell^{r}+1}{\ell+1}N(E/\F_{\ell})$.
\item[{\rm (iii)}] $g(F)\le (g(E)-1)\frac{\ell^{r}+1}{\ell+1}+\frac r2\left(\frac{\ell^{r}+1}{\ell+1}-1\right)+1.$
\end{itemize}
\end{lemma}
%
%
%
\begin{proof}
Let us outline the idea behind the proof. First we choose a place $Q$
of $E$ of odd degree $r$ and consider the constant extension
$\F_q\cdot E$. Then $Q$ remains a place of degree $r$ in $\F_q\cdot E$
since $r$ is odd (see \cite[Theorem 1.5.2(iii)(a)]{NX01}). We take the
narrow ray class field of $ \F_q \cdot E$ modulo $Q$ and then descend
to a subfield $K$ which is the fixed field of a certain subgroup of
the Galois group. This is done to ensure that the rational places of
$E$ which can be regarded as rational places of $\F_q \cdot E$ split
completely in $K/\F_q \cdot E$ (note that $K$ may not be a cyclic
extension over $\F_q \cdot E$). To obtain a cyclic extension over
$\F_q \cdot E$, we need to descend to a further subfield of $K$, which
will be our claimed function field $F$.  The reason why we use a place
$Q$ of odd degree is that, in the case of odd $r$, the narrow ray
class group of $\F_q \cdot E$ modulo $Q$ is a cyclic Galois extension
over its Hilbert class field. In the end, we can construct our desired
function field such that it is a cyclic extension over $\F_q \cdot E$.

Put $E_1:=E$ and consider the constant field extension $E_2:=\F_q\cdot E_1$. Then $\infty$ remains a rational place in $E_2$ and $Q$ remains a place of degree $r$ in $E_2$ as well.

Let $A_i$ be the ring in $E_i$ defined by
\[A_i:=\{x\in E_i:\; \nu_P(x)\ge 0 \ \mbox{ for all $P\not=\infty$}\}\]
and
let $H_i$ be the Hilbert class field of $A_i$ of $E_i$ with respect to $\infty$. Consider the narrow-ray class field $E_i^Q=H_i(\Lambda_{\phi}(\Q))$ where $\Q$ is the ideal corresponding to place $Q$. Then we can identify Gal$(E_{i}^Q/E_{i})$ with $\Cl^+_{{Q}}(A_{i})$.

 Now let
$K$ be the subfield of the extension $E_{2}^Q/E_2$ fixed by the subgroup
$G=\F_q^*\cdot\Cl^+_{{\mQ}}(A_1)$ of $\Cl^+_{{Q}}(A_{2})$.
We have
$$|G|=\frac{|\F_q^*|\cdot|\Cl^+_{{Q}}(A_1)|}{|\F_q^*\cap \Cl^+_{{Q}}(A_1)|}=\frac{(\ell^2-1) \cdot (\ell^r-1) h_{E_1}}{\ell-1 } = (\ell+1)(\ell^{r}-1)h_{E_1}$$
and so
\begin{equation}\label{eq:4.13}
[K:E_{2}]=\frac{|\Cl^+_{{D}}(A_{2})|}{|G|}=\frac{(q^r-1) h_{E_2}}{|G|}
=
\frac{\ell^{r}+1}{\ell+1}\times\frac{h_{E_2}}{h_{E_1}}.\end{equation}

Let $P_{\infty}$ be a place of $K$ lying over $\infty$. Then
the inertia group of $P_{\infty}$ in the extension $E_2^Q/K$ is $\F_q^*
\cap G$, and so the ramification index $e(P_{\infty}|\infty)$ of $P_{\infty}$
over $\infty$ is given by
$$e(P_{\infty}|\infty)=\frac{|\F_q^*|}{|\F_q^*\cap G|}=\frac
{|\F_q^*\cdot G|}{|G|}=\frac{|\F_q^*\cdot \Cl^+_{{Q}}(A)|}{|G|}=1,$$
i.e., $\infty$ is unramified in $K/E_2$.

Let $R$ be a place of $K$ lying over $Q$. Since the inertia group of $Q$ in
$E_2^Q/E_2$ is $(A_{2}/{\Q})^*$ by the theory of narrow ray class fields, the inertia group of $R$ in
$E_2^Q/K$ is $(A_{2}/{\mQ})^{*}\cap G=\F_q^*\cdot(A/{\Q})^*$.
Thus, the ramification index $e(R|Q)$ of $R$ over $Q$ is given by
\begin{equation}\label{eq:4.14}
e(R|Q)=\frac{|(A_{2}/{\Q})^{*}|}{|\F_q^*\cdot(A_1/{\Q})^{*}|}=\frac{|(A_{2}/{\Q})^{*}|
\cdot|\F_q^*\cap(A_1/{\Q})^{*}|}{|\F_q^*|\cdot|(A_1/{\Q})^{*}|}=
\frac{(q^r-1)(\ell-1)}{(q-1)(\ell^r-1)}=
\frac{\ell^{r}+1}{\ell+1}.
\end{equation}
Since $\infty,Q$ are the only ramified places in $E_2^Q/E_2$, and $\infty$ is unramified in $K/E_2$, we conclude that the place $Q$ is the only ramified place in $K/E_2$ with ramification index $(\ell^r+1)/(\ell+1)$.

Now, all $\F_{\ell}$-rational places of $E_1$ can be viewed as $\F_q$-rational places of $E_2$ and furthermore they split completely in $K$. This is because for a rational place $P$ of $E_1$ with $P\neq \infty,Q$, from Proposition \ref{prop:CFT} and our construction, it follows that the Frobenius automorphism of $P$ is contained in the subgroup $\Gal(E_2^Q/K)$.
Therefore, the Frobenius automorphism of $P$ in the extension $K/E_2$ is the identity, and therefore $P$ must split completely in $K/E_2$.

 Since the decomposition group of $Q$ in $E_2^Q/E_2$ is isomorphic to the cyclic group $(A/{\Q})^*$, the decomposition group of $Q$ in $K/E_2$ is cyclic as well. The inertia group of $Q$, which is a subgroup of the decomposition group of $Q$, has order $\frac{\ell^{r}+1}{\ell+1}$ by (\ref{eq:4.14}). Thus, the Galois group $\Gal(K/E_2)$ contains a cyclic subgroup of order $\frac{\ell^{r}+1}{\ell+1}$. This implies that there exists a subfield $F$ of $K/E_2$ such that $\Gal(F/E_2)$ is a cyclic group of order $\frac{\ell^{r}+1}{\ell+1}$.
 It is clear that all $\F_{\ell}$-rational places of $E_1$ split completely in $F$ as well. Hence
\[N(F/\F_q)\ge [F:E_2]N(E_1/\F_{\ell})= \frac{\ell^{r}+1}{\ell+1}N(E_1/\F_{\ell}) \ .\]
Moreover, the place $Q$ is the only ramified place in $F/E_2$ (since it is the only ramified place in $K/E_2$) and it is tamely ramified with the ramification index at most $[F:E_2]$.
Hence, we can apply the Hurwitz genus formula to
the extension $F/E_2$ and get
\[2g(F)-2  \le (2g(E_2)-2)[F:E_2]+r([F:E_2]-1).\]
The desired result follows from the fact that $g(E_2)=g(E_1)$.
\end{proof}

The following theorem provides the family of function fields that we required to construct our folded algebraic geometric codes in Theorem \ref{thm:assum-list-size}.

\begin{theorem}\label{2.1} Let $\ell$ be a prime power and let $q=\ell^2$. Assume that there is a family $\{E/\F_{\ell}\}$  of function fields such that $g(E)\rightarrow\infty$ and $N(E/\F_{\ell})/g(E)\rightarrow A$ for a positive real $A$. Then for any odd integer $r$ with $r>\log(2+7g(E))/\log(\ell)$,  there exists a function field $F/\F_{q}$  such that $F$ is a finite extension of $\F_q\cdot E$ of degree $e:=(\ell^r+1)/(\ell+1)$ and
\begin{itemize}
\item[{\rm (i)}] $g(F)\rightarrow\infty$ and $g(F)\le (g(E)-1)e+r(e-1)/2+1$.
\item[{\rm (ii)}] $N(F/\F_q)\ge eN(E/\F_{\ell})$.
\item[{\rm (iii)}] $F/(\F_q\cdot E)$ is a cyclic Galois extension of degree $e$.
\end{itemize}
In particular, we have $\liminf_{g(F)\rightarrow\infty}N(F/\F_q)/g(F)\ge A/(1+c)$ if $r/g(E)\rightarrow 2c$ for a constant $c\ge 0$.
\end{theorem}
\begin{proof} The result follows directly from Lemma \ref{2.1.3} and the fact that there exists a place of degree $r$ in $E$ as long as $r>\log(2+7g(E))/\log(\ell)$ (see \cite[Corollary 5.2.10]{stich-book}). This completes the proof.
\end{proof}

\subsection{Chebotarev Density Theorem}
Given Theorem \ref{2.1}, to show that the family of function fields with property (P3) exists, it remains to find a large set $S$ of places of $\PP_F$ satisfying (P3)(ii). In order to accomplish this task, we need the explicit form of the Chebotarev Density Theorem.

Let $F/L$ be a Galois extension of degree $e$ of function fields over $\F_q$. Assume that $\F_q$ is the full constant field of both $F$ and $L$. Let $t$ be a separating transcendence element over $\F_q$. Let $d=[L:\F_q(t)]$.

For a place $Q$ of $F$ lying over $P$ of $L$, let $\left[\frac{F/L}{Q}\right]$ be the Frobenius of $Q$. Then for any $\Gs\in \Gal(F/L)$, the Frobenius of $\Gs(Q)$ is $\Gs\left[\frac{F/L}{Q}\right]\Gs^{-1}$. Thus, the conjugacy class  $\left\{\Gs\left[\frac{F/L}{Q}\right]\Gs^{-1}:\; \Gs\in \Gal(F/L)\right\}$ is determined by $P$. We denote this conjugacy by $\left[\frac{F/L}{P}\right]$.

Fix a conjugacy class $C$ of $\Gal(F/L)$, let $M_h(C)$ denote the number of places $P$ of degree $h$ in $L$ that are unramified in both $F/L$ and $L/\F_q(t)$  such that $\left[\frac{F/L}{P}\right]=C$. Then we have the following result \cite[Proposition 6.4.8]{FJ08} and \cite{MS94}.

\begin{theorem}[Chebotarev Density Theorem]\label{2.2}
One has
\begin{equation}
\left|M_h(C)-\frac{|C|}{e h}q^h\right|\le \frac{2|C|}{e h}(e+g_F)q^{h/2}+e(2g_L+1)q^{h/4}+g_F+de,\end{equation}
where $g_F$ and $g_L$ denote the genera of $F$ and $L$, respectively.
\end{theorem}

Finally, we are able to show existence of a family of function fields with (P3).

\begin{theorem}\label{2.3} There exists a family of function fields with {\rm (P3)}. More precisely speaking, we have the following result.

 Let $\ell$ be a square prime power and let $\{E/\F_{\ell}\}$ be the Garcia-Stichtenoth tower given in {\rm \cite{GS95}}. Let $\{F/\F_q\}$ be the family of function fields constructed in Theorem {\rm \ref{2.1}}. Put $L:=\F_q\cdot E$, $n:=N(E/\F_{\ell})$ and denote by $\Gs$ a generator of $\Gal(F/L)$. Let $r= 2\lceil n/(\sqrt{\ell}-1)\rceil+1$ and $h=3r$. Then  there exists a set $S$ of places of $L$ such that
 \begin{itemize}
\item[{\rm (i)}] $|S|\ge q^r$ and $\deg(R)=h$ for all places $R$ in $S$.
\item[{\rm (ii)}]  For every place $R$ in $S$, there is a unique place $P$ of $F$ of degree $eh$ lying over $R$ and $\left[\frac{F/L}{P}\right]=\Gs^{-1}$.
\item[{\rm (iii)}] $\liminf N(F/\F_q)/g(F)\ge (\sqrt{\ell}-1)/2=(q^{1/4}-1)/2$.
 \end{itemize}
\end{theorem}
\begin{proof} Let $\{E/\F_{\ell}\}$ be the well-known Garcia-Stichtenoth tower \cite{GS95}. Then one has $N(E/\F_{\ell})/g(E)\rightarrow\sqrt{\ell}-1$ with $g(E)\rightarrow\infty$ and $[E:\F_{\ell}(t)]\le g(E)$ for a separating transcendence element over $\F_{\ell}$. Thus, $d=[L:\F_q(t)]\le g(E)=g(L)$. By our choice of parameters $h, r$, we find that
\[\frac{1}{e h}q^h-\left(\frac{2}{e h}(e+g_F)q^{h/2}+e(2g_L+1)q^{h/4}+g_F+de\right)\ge q^r.\]
By Theorem \ref{2.2}, there exists a set $S$ of places of $L$ with  $|S|\ge q^r$ such that $\deg(R)=h$ and $\left[\frac{F/L}{R}\right]=\Gs^{-1}$  for every place $R$ in $S$. Let $P$ be a place of $F$ lying over $R$. The Frobenius $\Gs^{-1}$ of $P$   belongs to $\Gal(F/Z)$, where $Z$ is the decomposition field of $P$ in $F/L$. Since the order of $\Gs^{-1}$ is $e=[F:L]$, we must have $Z=L$ and hence the relative degree $f(P|R)$ is $e$. So $P$ is the only place of $F$ lying over $R$.

Since $r/g(E)\rightarrow 2$, we have $\liminf N(F/\F_q)/g(F)\ge (\sqrt{\ell}-1)/2$ by Theorem \ref{2.1}.
\end{proof}

\section{Encoding and Decoding}\label{sec:ED}
We have not considered encoding and decoding of the folded algebraic geometric codes constructed in Section \ref{sec:ALD}. This section is devoted to the computational aspects of encoding and decoding of our folded codes.
\subsection{Encoding}
Let us consider the folded algebraic geometric code given in the proof of Theorem \ref{thm:assum-list-size}, where the divisor $D$ is $l\sum_{P_{\infty}|\infty, P_{\infty}\in\PP_F}P$ and the Riemann-Roch space is $\cL(lD)$. To encode, we assume that  $le>2g-1$ and there is an algorithm to find a basis $\{z_1,z_2,\dots,z_k\}$ of $\cL(lD)$ with $k= le-g+1$.

Furthermore, we assume that, for every point $P_i^{\sigma^j}$ and each function $f$ with $\nu_{P_i^{\sigma^j}}(f)\ge 0$, there is an efficient algorithm to evaluate $f$ at $P_i^{\sigma^j}$, i.e., find $f(P_i^{\sigma^j})$. For a function $f$ and a rational place $P$ with $\nu_P(f)\ge 0$,  the algorithm of evaluating  $f$ at  $P$ consists of
\begin{itemize}
\item [(i)] Finding a local parameter $t$ at $P$ (recall that a function $t$ is called a local parameter at  $P$ if $\nu_P(t)=1$).
\item [(ii)] Finding the unique element $\Ga\in\F_q$ such that $\nu_P\left(\frac{f-\Ga}t\right)\ge 0$ (note that this unique element $\Ga$ is equal to  $f(P)$).
\end{itemize}

\subsection{Decoding}
As we have seen, encoding is easy as long as we have an efficient algorithm to compute a basis of the Riemann-Roch space and evaluation at rational places. However, we need some further work for decoding.

The idea of decoding is to solve the equation (\ref{eq:alg-eqn})
through local expansions at a point. Let us briefly introduce local
expansions first. The reader may refer to \cite[pages 5-6]{NX01} for the detailed
result on local expansions.
Let $F/\F_q$ be a function field and let
$P$ be a rational place.
For a nonzero function $f\in F$ with $\nu_P(f)\ge
v$, we have
$\nu_P\left(\frac f{t^v}\right)\ge 0.$
Put
$a_v=\left(\frac f{t^v}\right)(P),$
i.e., $a_v$ is the value of the function $f/t^v$ at $P$.  Note that the function $f/t^v-a_v$ satisfies
$\nu_P\left(\frac f{t^v}-a_v\right)\ge 1,$
hence we know that
$ \nu_P\left(\frac {f-a_vt^v}{t^{v+1}}\right)\ge 0.$
Put
$a_{v+1}=\left(\frac{f-a_vt^v}{t^{v+1}}\right)(P).$
Then  $\nu_P(f-a_vt^v-a_{v+1}t^{v+1})\ge v+2$.

Assume that we have obtained a sequence $\{a_r\}_{r=v}^m$ ($m>v$)
of elements of $\F_q$ such that
$\nu_P(f-\sum_{r=v}^ka_rt^r)\ge k+1$
for all $v\le k\le m$.
Put
$a_{m+1}=\left(\frac{f-\sum_{r=v}^ma_rt^r}{t^{m+1}}\right)(P).$
Then  $\nu_P(f-\sum_{r=v}^{m+1}a_rt^r)\ge m+2$.
In this way we continue our construction of the $a_r$. Then we obtain
an infinite sequence $\{a_r\}_{r=v}^{\infty}$ of elements of $\F_q$ such
that
$
\nu_P(f-\sum_{r=v}^ma_rt^r)\ge m+1
$
for all $m\ge v$.
We summarize the above  construction in the formal expansion
\begin{equation}
\label{eqn_1.1}
f=\sum_{r=v}^{\infty}a_rt^r,
\end{equation}
which is called the { local expansion} of $f$ at $P$.

 It is clear
that local expansions of a function depend on choice of the local
parameters $t$. Note that if a power series
$\sum_{i=v}^{\infty}a_it^i$ satisfies $\nu_P(f-\sum_{i=v}^ma_it^i)\ge
m+1$ for all $m\ge v$, then it is a local expansion of $f$. The above procedure shows that finding a local expansion at a rational place is very efficient  as long as the computation of evaluations of functions at this place is easy.

The following fact plays an important role in our decoding.

\begin{lemma}\label{lem:conj-expansion}
Let $F/\F_q$ be a function field and let $\s\in {\rm Aut}(F/\F_q)$ be an automorphism. Let $P, P^{\s^{-1}}$ be two distinct rational places. Assume that $t$ is a common local parameter of $P$ and  $P^{\s}$, i.e., $\nu_P(t)=\nu_{P^{\s}}(t)=1$ such that $t^{\s}=t$. Suppose that $f\in F$ has a local expansion $\sum_{i=0}^{\infty}a_it^i$ at $P^{\s}$ for some $a_i\in \F_q$, then the local expansion of $f^{\s^{-1}}$ at $P$ is $\sum_{i=0}^{\infty}a_it^i$.
\end{lemma}
\begin{proof} By the definition of local expansion,  we have $\nu_{P^{\s}}\left(f-\sum_{i=0}^{m}a_it^i\right)\ge m+1$ for all $m\ge 0$. This gives $\nu_{(P^{\s})^{\s^{-1}}}\left((f-\sum_{i=0}^{m}a_it^i)^{\s^{-1}}\right)=\nu_{P}\left(f^{\s^{-1}}-\sum_{i=0}^{m}a_it^i\right)\ge m+1$ for all $m\ge 0$. The desired result follows.
\end{proof}

Now let $F$ be the function field constructed in Corollary \ref{2.3}. Then it satisfies the Property (P3). Assume that $U \neq\infty$ is a rational place of $E$ and $t\in E$ is a local parameter at $U$. Then $U$ can be viewed as an $\F_q$-rational point of $L=\F_q\cdot E$. Moreover, $U$ splits completely in $F/L$. We may assume that all rational places of $F$ lying over $U$ are $P,P^{\s},\dots,P^{\s^{e-1}}$,  where $\s$ is a generator of ${\rm Gal}(F/L)$.  It is clear that $t$ is a common local parameter of $P,P^{\s},\dots,P^{\s^{e-1}}$. Furthermore, we have $t^{\s}=t$ since $t\in E\subset L$.

To solve for the functions $f$ that satisfy the algebraic equation \eqref{eq:alg-eqn}, let us assume that $f=\sum_{i=1}^kf_iz_i$ for some $f_i\in\F_q$, where $k=le-g+1$ is the dimension of $\cL(lD)$. Solving for $f$  in  (\ref{eq:alg-eqn}) is equivalent to finding $\{f_i\}_{i=1}^k$. Assume that the local expansion of $z_i$ at $P^{\s^j}$ is given by $\sum_{h=0}^{\infty}\Ga_{ijh}t^h$.
 Then by  Lemma \ref{lem:conj-expansion}, $z_i^{\s^{-j}}$ have the local expansion $\sum_{h=0}^{\infty}\Ga_{ijh}t^h$ at $P$. Thus, $f^{\s^{-j}}$ has the local expansion $\sum_{i=1}^k\sum_{h=0}^{\infty}\Ga_{ijh}t^h$ at $P$. Furthermore assume that $A_i$ have  local expansions $\sum_{j=0}^{\infty}a_{ij}t^j$ at $P$ for $0\le i\le s$.  Substitute these local expansions in Equation \eqref{eq:alg-eqn}, we obtain an equation
\begin{equation}\label{eq:coeff-power}
c_0(f_1,f_2,\dots,f_k)+c_1(f_1,f_2,\dots,f_k)t+\dots+c_i(f_1,f_2,\dots,f_k)t^i+\dots=0,
\end{equation}
where $c_i(f_1,f_2,\dots,f_k)$ is a linear combination of $f_1,f_2,\dots,f_k$ for all $i\ge 0$.
Thus, each of the coefficients of the above power series (\ref{eq:coeff-power}) must be zero. This produces infinitely many linear equations $c_i(f_1,f_2,\dots,f_k)=0$ for $i\ge 0$ in variables $f_1,f_2,\dots,f_k$. This system of infinitely many linear equations is equivalent to the system
\begin{equation}\label{eq:coeff-sol}
c_i(f_1,f_2,\dots,f_k)=0 \qquad \mbox{for $i=0,1,\dots,le$}
\end{equation}
due to the fact that  $A_0+A_1f+\cdots+A_sf^{\s^{-(s-1)}}\in\cL(lD)$   and the following simple claim.
\begin{lemma}\label{lem:expansion-equal}
If $x$ is an element in $\cL(lD)$ and has a local expansion $\sum_{i=le+1}^{\infty}\Gl_it^i$ for some $\Gl_i\in\F_q$, then $x$ is identical to $0$ if $\Gl_i=0$ for all $i\le le$.
\end{lemma}
\begin{proof} By the local expansion of $x$, we know that $x$ belongs to $\cL(lD-(le+1)P)$. The desired result follows from the fact that $\deg(lD-(le+1)P)=-1<0$.
\end{proof}
The equation system (\ref{eq:coeff-sol}) has $le+1$ equations and contains $k=le-g+1$ variables. Theorem \ref{thm:assum-list-size} guarantees that this system has at most $O(N^{(\sqrt{\ell}-1)s})$ solutions.

Given the discussion of encoding and decoding, we rewrite Theorem \ref{thm:assum-list-size} as the main result of this paper.
\begin{theorem} [Main]\label{thm:main} For any small $\Ge>0$ and a real $0<R<1$, one can construct a   folded algebraic geometric code over alphabet size $(1/\Ge)^{O(1/\Ge^2)}$ with rate $R$ and decoding radius $\tau=1-R-\Ge$ such that the length of the code tends to $\infty$ and is independent of $\Ge$. Moreover, the code is deterministically list decodable with a list size $O(N ^{1/\epsilon^2} )$.

Given a
polynomial amount of pre-processed information about the code, the
algorithm essentially consists of solving two linear systems over
$\F_q$, and thus runs in deterministic polynomial time.
\end{theorem}
\begin{proof}
 In Theorem \ref{thm:assum-list-size},  choose $s\approx 1/\Ge$ and $m\approx 1/\Ge^2$ and $q\approx 1/\Ge^4$, the error fraction $\tau $ given in Theorem \ref{thm:assum-list-size} is  $1-R-\Ge$. The alphabet size of the folded code is $q^m$, which is $(1/\Ge)^{O(1/\Ge^2)}$ and the list size is $O(N^{(\sqrt{\ell}-1)s})=O(N ^{1/\epsilon^2} )$.
\end{proof}

\subsection{Computing A Basis of Riemann-Roch Space}
Both encoding and decoding described earlier depend on an algorithm to find a basis of the Riemann-Roch space  $\cL(lD)$. We divide this job into two steps. The first step is to  find an explicit equation defining  our function field $F$ constructed in Section \ref{FF}  through class field method. The second step is to  find a basis of our Riemann-Roch space based on the equation form Step 1.

In \cite{Hess02}, a polynomial algorithm of finding a basis of a Riemann-Roch space is given based on an explicit equations of the associated function field. If $F$ is of the form $\F_q(x,y)$ with a defining equation
\begin{equation}\label{eq:defining equation}
y^h+a_1(x)y^{h-1}+\cdots+a_{h-1}(x)y+a_h(x)=0
 \end{equation}
 with $a_i(x)\in\F_q[x]$, then \cite{Hess02} describes an algorithm with polynomial time in $h$, the divisor degree $le$ and $\Delta$, where $\Delta$ is the largest degree of $a_1(x),a_2(x),\dots,a_h(x)$ in (\ref{eq:defining equation}). Thus, if we can find an equation defining the field $F$ with $\Delta$ being a polynomial in the code length $N$, then \cite{Hess02} provides an polynomial algorithm in the code length $N$ to determine a basis of $\cL(lD)$.

Thus, to get a polynomial time encoding and decoding for our folded algebraic code, it is sufficient to obtain  polynomial time algorithms for
\begin{itemize}
\itemsep=1ex
\item[(i)] finding a defining equation (\ref{eq:defining equation}) of $F$ such that $\Delta$ is a polynomial in code length $N$;
\item[(ii)]  computing evaluations of  functions at rational places.
\end{itemize}

Part (ii) is usually easier. The key part is to find a defining equation of the underlying function field.  To see this, we start with the function field $E$ defined by the Garcia-Stichtenoth tower. Then one has $[E:\F_{\ell}(x)]\le N(E/\F_{\ell})$. Moreover, $E/\F_{\ell}(x)$ is a separable extension, thus there exists $\Gb\in E$ such that $E=\F_{\ell}(x,\Gb)$. Consequently, we have $L=\F_q\cdot E=\F_{q}(x,\Gb)$.

The paper \cite{DF12} describes a method to find an element $\Ga$ of $F$ such that $F=L(\Ga)$. Thus, $F=\F_q(x,\Ga,\Gb)$. Now the problem is how to find an element $y\in F$ such that $F=\F_q(x,\Ga,\Gb)=\F_q(x,y)$ with the defining equation given in (\ref{eq:defining equation}) and the maximum degree $\Delta$ is a polynomial in $N$.

We summarize what we discussed above into an open problem.
\begin{open}
Find a  polynomial time algorithm to construct an explicit equation {\rm (\ref{eq:defining equation})} of the function field $F$ given in Theorem {\rm \ref{2.3}} and compute a basis of the Riemann-Roch space efficiently.
\end{open}

\providecommand{\bysame}{\leavevmode\hbox to3em{\hrulefill}\thinspace}
\providecommand{\MR}{\relax\ifhmode\unskip\space\fi MR }
\providecommand{\MRhref}[2]{%
  \href{http://www.ams.org/mathscinet-getitem?mr=#1}{#2}
}
\providecommand{\href}[2]{#2}


\end{document}